\documentclass[11pt]{amsart}

\usepackage{verbatim, latexsym, amssymb, amsmath}

\theoremstyle{plain}
\newtheorem{theorem}{Theorem}
\newtheorem*{theorem*}{Theorem}
\newtheorem{lemma}[theorem]{Lemma}

\newtheorem{proposition}[theorem]{Proposition}

\theoremstyle{remark}
\newtheorem{remark}{Remark}

\theoremstyle{definition}
\newtheorem*{definition}{Definition}

\title[On perturbations of the Schwarzschild anti-de Sitter spaces]{On perturbations of the \mbox{Schwarzschild anti-de Sitter} spaces of positive mass}
\author{Lucas C. Ambrozio}
\address{Instituto de Matem\'atica Pura e Aplicada (IMPA) \\ Estrada Dona Castorina 110 \\ 22460-320 Rio de Janeiro \\ Brazil; lambroz@impa.br}
\thanks{The author was supported by FAPERJ and CNPq-Brasil}

\begin{document}
 
\begin{abstract}
In this paper we prove the Penrose inequality for metrics that are small perturbations of the Schwarzschild anti-de Sitter metrics of positive mass. We use the existence of a global foliation by weakly stable constant mean curvature spheres and the monotonicity of the Hawking mass. \\
\end{abstract}

\maketitle

\section{Introduction}

\indent Let $(M^3,g)$ be a complete Riemannian three-manifold, possibly with boundary, with exactly one end. Assume that the complement of a compact subset of $M$ is diffeomorphic to $\mathbb{R}^3$ minus a ball. Roughly speaking, $(M,g)$ is called \textit{asymptotically hyperbolic} when the metric $g$ decays sufficiently fast to the hyperbolic metric when expressed in the spherical coordinates induced by this chart. If an asymptotically hyperbolic manifold $(M,g)$ has scalar curvature $R\geq -6$ and its boundary $\partial M$ is empty or has mean curvature $H\leq 2$, there is a well-defined geometric invariant called the \textit{total mass} of $(M,g)$, a non-negative number $m$ that is zero if and only if $(M,g)$ is isometric to the hyperbolic space. This is the content of the Positive Mass Theorem in the asymptotically hyperbolic setting, proved with spinorial methods by X. Wang \cite{W} when $\partial M$ is empty and by P. Chru\'{s}ciel and M. Herzlich \cite{CH} under weaker asymptotic conditions. \\
\indent If an asymptotically hyperbolic manifold $(M,g)$ has scalar curvature $R\geq -6$ and its boundary $\partial M$ is a \textit{connected minimal surface} that is \textit{outermost}, i.e., if there are no other closed minimal surfaces in $M$, then it is conjectured that the total mass $m$ of $(M,g)$ is related to the area of $\partial M$ by the following inequality:
\begin{equation}\label{enunciadopenrose}
  \left(\frac{|\partial M|}{16\pi}\right)^{\frac{1}{2}} + 4\left(\frac{|\partial M|}{16\pi}\right)^{\frac{3}{2}} \leq m.
\end{equation}
\indent This statement is known as the \textit{Penrose Conjecture} in the asymptotically hyperbolic setting. We refer the reader to the surveys \cite{BC} and \cite{M}, where he will find a comprehensive discussion on these types of inequalities in various settings and also on the physics behind these conjectures. In the original \textit{asymptotically flat} setting, the Penrose Conjecture was proved with different techniques by G. Huisken and T. Ilmanen \cite{HI} and by H. Bray \cite{B2}.\\

\indent The Penrose Conjecture contains also a rigidity statement. There are important models, known as the \textit{Schwarzschild anti-de Sitter spaces of positive mass}, that satisfy the equality in (\ref{enunciadopenrose}). They are obtained as spherically symmetric metrics $g_m$ on $M=[0,+\infty)\times S^2$ with constant scalar curvature $-6$, where the parameter $m$ is a positive real number
that coincides with the total mass of $(M,g_m)$ above discussed (see Section \ref{Secao2} for more details). The Penrose Conjecture also asserts that the Schwarzschild anti-de Sitter spaces of positive mass are the unique asymptotically hyperbolic manifolds, with scalar curvature $R\geq-6$ and an outermost minimal boundary, that satisfy the equality in (\ref{enunciadopenrose}). \\

\indent A special feature of the Schwarzschild anti-de Sitter spaces of positive mass is that they are foliated by constant mean curvature spheres which are \textit{weakly stable}. This kind of foliation of an asymptotically hyperbolic manifold is interesting, among other reasons, because of a monotonicity result observed by H. Bray in \cite{B1}: if $R\geq-6$, the so-called Hawking mass functional, defined for closed surfaces $\Sigma$ in $(M,g)$ by
\begin{equation}\label{mh}
m_{H}(\Sigma):=\sqrt{\frac{|\Sigma|}{16\pi}}\left( 1-\frac{1}{16\pi}\int_{\Sigma}( H^2 - 4 )d\Sigma \right),
\end{equation}
\noindent is monotone non-decreasing along these foliations in the direction of increasing area of the leaves. \\
\indent If $\partial M$ is a minimal surface in $(M,g)$, its Hawking mass is exactly the left-hand side of (\ref{enunciadopenrose}). This suggests an approach to prove the Penrose inequality (\ref{enunciadopenrose}), at least for the class of asymptotically hyperbolic manifolds $(M,g)$, with $R\geq -6$ and an outermost minimal boundary, admitting this kind of foliation starting at $\partial M$ 
and sweeping out all $M$: Inequality (\ref{enunciadopenrose}) would follow if the Hawking mass of the leaves near the infinity converges to the total mass of $(M,g)$. \\

\indent There are asymptotically hyperbolic manifolds admitting unique foliations by weakly stable CMC spheres \textit{near the infinity}. Results in this direction were proved by R. Rigger \cite{R} and
 then by A. Neves and G. Tian in a quite general setting (see \cite{NT1} and \cite{NT2}). See also \cite{MP} for other results of this nature in asymptotically hyperbolic settings. For the original asymptotically flat setting, see the pioneering work of G. Huisken and S.T. Yau \cite{HY}. \\
\indent In the present paper, we consider metrics on $M=[0,+\infty)\times S^2$ that are small \textit{global} perturbations of the Schwarzschild anti-de Sitter metric $g_m$ in a suitable sense and show that the foliation constructed in \cite{NT2} outside a compact set can be extended up to $\partial M$. Following the argument outlined above, we then show the Penrose inequality is true for these asymptotically hyperbolic manifolds. \\
\indent In order to state our results more precisely, let us introduce some terminology. Given $m>0$, the perturbations of $g_m$ we consider belong to the space $\mathcal{M}(M,m)$ of metrics $g$ of class $C^3$ on $M:=\{p=(s,x)\in[0,+\infty)\times S^2\}$ such that $\partial M$ is a minimal surface in $(M,g)$ and
\begin{equation*}
 d(g,g_m) := \sup_{p\in M}\left(\sum_{i=0}^{3}\exp(4s)\|(\nabla^{m})^i(g-g_m)\|_{g_{m}}(p)\right) < +\infty,
\end{equation*}
\noindent see Section \ref{Secao2} for more details. All these metrics are asymptotically hyperbolic with total mass $m$ in the sense of \cite{W}. \\
\indent Our main results can be summarized as follows (see Theorem \ref{colagem} and Theorem \ref{penrose}):
\begin{theorem*}
 Let $m>0$. There exists $\epsilon>0$ such that for every metric $g\in\mathcal{M}(M,m)$ with $d(g,g_m)<\epsilon$ the following statements hold:
  \begin{itemize}
    \item[$i)$] There exists a foliation $\{\Sigma_t\}_{t\in[0,+\infty)}$ of $(M,g)$ by weakly stable constant mean curvature spheres such that $\Sigma_{0}=\partial M$ is an outermost minimal surface.
    \item[$ii)$] (The Penrose inequality). If the scalar curvature of $g$ is greater than or equal to $-6$, then
    \begin{equation*}
      \left(\frac{|\partial M|}{16\pi}\right)^{\frac{1}{2}} + 4\left(\frac{|\partial M|}{16\pi}\right)^{\frac{3}{2}} \leq \lim_{t\rightarrow+\infty}m_{H}(\Sigma_{t}) = m, 
    \end{equation*}
    \noindent with equality if and only if $(M,g)$ is isometric to the Schwarzschild anti-de Sitter space
of mass $m$.
  \end{itemize}
\end{theorem*}

\indent The paper is organized as follows. In Section \ref{Secao2}, we describe the geometric properties
of the models and define the class of perturbations of $(M,g_m)$ we are going to work with. In Section \ref{Secao3}, we construct foliations of compact regions of $M$ by weakly stable CMC spheres which begin at the minimal boundary. In Section \ref{Secao4}, we recall A. Neves and G. Tian's result in \cite{NT2} about existence and uniqueness of foliations near the infinity of certain asymptotically hyperbolic manifolds by weakly stable CMC spheres, stating a version adapted to the space of metrics we consider. In Section \ref{Secao5}, we calculate the limit of the Hawking mass of the leaves of that foliation. In Section \ref{Secao6}, we argue that the foliations constructed before match on their overlap if the perturbation is small enough and show that the obtained foliation has the properties described in the above theorem. In Section \ref{Secao7}, we prove the Penrose inequality for these small perturbations following the monotonicity argument outlined above.  \\

\indent We remark that in the asymptotically hyperbolic setting there is also another form of the Penrose Conjecture where the boundary corresponds to some $H=2$ surface (see \cite{BC} and \cite{M}). In the end of the paper we briefly discuss it and explain the modifications of the previous results that establish 
it for small perturbations (see Theorem \ref{colagem2} and Theorem \ref{penrose2}). \\

\indent The approach to the Penrose Conjecture involving the monotonicity of the Hawking mass for a family of surfaces that interpolates the outermost boundary and the infinity was originally suggested 
in the asymptotically flat setting by R. Geroch \cite{G}, who proposed the inverse mean curvature flow to produce such family. This program was successfully implemented by G. Huisken and T. Ilmanen \cite{HI}. In the asymptotically hyperbolic setting, however, there are serious difficulties in using this approach, see the work of A. Neves \cite{N}. \\
\indent In the asymptotically hyperbolic setting, the Penrose Conjecture in its full generality is still an open problem. Using Bray's monotonicity as in \cite{B1}, J. Corvino, A. Gerek, M. Greenberg and B. Krummel \cite{CGGK} proved the conjecture for asymptotically hyperbolic manifolds that are isometric to a Schwarzschild anti-de Sitter space of positive mass outside a compact set under restrictive hypotheses on the behavior of its isoperimetric surfaces, among them the assumption that the coordinate spheres near infinity are isoperimetric surfaces. This hypothesis has been proved to be always verified for compact perturbations of the models by O. Chodosh \cite{C}, whose recent work has also showed that isoperimetric surfaces exists for all sufficiently large volumes in asymptotically hyperbolic manifolds with scalar curvature $R\geq -6$. Finally, F. Gir\~ao and L. L. de Lima proved both cases of the conjecture and the
higher dimensional analogues of them for another class of asymptotically hyperbolic manifolds, those that are certain graphical hypersurfaces of the hyperbolic space (see \cite{GL1} and \cite{GL2}). For related work on Penrose type inequalities in asymptotically hyperbolic settings, we refer the reader to \cite{BO}, \cite{DGS} and \cite{LN}. \\
\indent As pointed out by one of the referees, our result may be also conceived as a ``local'' Penrose inequality and may be related to the dynamical stability of the Schwarzschild anti-de Sitter spaces of positive mass as solutions to the Einstein vacuum equations with negative cosmological constant (for further information on this subject, see \cite{FMR} and \cite{HW}). \\

\indent \textit{Acknowledgements.} I am grateful to my Ph.D advisor at IMPA, Fernando Cod\'a Marques, for his constant encouragement and advice. I would like to thank the referees for their many useful comments. I was supported by FAPERJ and CNPq-Brasil.
 

\section{The Schwarzschild anti-de Sitter spaces and its perturbations} \label{Secao2}

\indent Let $m$ be a real number. Let $\rho_m : (r_0,+\infty) \rightarrow \mathbb{R}$ be the function given by $\rho_m(r)=\sqrt{1+r^2-2m/r}$, where $r_0=r_{0}(m)$ is the unique positive zero of $\rho_m$ (if $m>0$) or $0$ (if $m\leq0$). Let $(S^2,g_0)$ be the standard round sphere of constant Gaussian curvature $1$ and let $dS^2$ denote its volume element. \\
\indent We call \textit{Schwarzschild anti-de Sitter space of mass $m$} the metric completion of the Riemannian mani\-fold $((r_0,+\infty)\times S^2,g_m)$, where using the natural $r$ coordinate the metric $g_m$ is written as 
$$g_m= \frac{dr^2}{\rho_{m}^2(r)} + r^2g_0 = \frac{dr^2}{1+r^2-\frac{2m}{r}} + r^2 g_0.$$
\indent Although the expression of $g_m$ in this coordinate system becomes singular at $r_0$, it can be proved that when $m>0$ the metric $g_m$ extends to a (smooth) Riemannian metric on $M=[r_0,\infty)\times S^2$. On the other hand a smooth extension is not possible if $m<0$. Notice also that if we let the parameter $m$ to be zero we recover the hyperbolic space (this can be easily seen by performing the coordinate change $r=\sinh s$). \\
\indent We will call \textit{coordinate spheres} the surfaces $S_{r}=\{r\}\times S^2\subset M$.
The following proposition describes the geometry of $(M,g_m)$ and of its coordinate spheres.
\begin{proposition}\label{geomAdSS}
  (Geometry of the Schwarzschild anti-de Sitter space of mass $m$)
  \begin{itemize} 
   \item[$i)$] The Ricci curvature of $g_m$ is given by
    $$Ric_m=(-2-\frac{2m}{r^3})\frac{1}{\rho_m^2(r)}dr^2+(-2+\frac{m}{r^3})r^2 g_0.$$
   \item[$ii)$] The scalar curvature of $g_m$ is constant and equal to $-6$. 
   \item[$iii)$] The coordinate spheres $S_r$ are totally umbilic surfaces with constant mean curvature given by
     $$ H_{m}(r) = \frac{2}{r}\sqrt{1+r^2-\frac{2m}{r}}.$$
   \item[$iv)$] The Hawking mass of all coordinate spheres is $m$.  
   \item[$v)$] The Jacobi operator of the coordinate sphere $S_r$ is given by
     $$ L_r = \frac{1}{r^2} \left(\Delta_0 + \left(2-\frac{6m}{r} \right)\right), $$
   where $\Delta_0$ is the Laplacian operator of the round sphere $(S^2,g_0)$.
  \end{itemize}
\end{proposition}
\begin{proof}
  A calculation in coordinates. 
\end{proof}
\indent Notice that when $m>0$ the Jacobi operator of $S_r$ is invertible except for $r=3m$. In any case, it is invertible when restricted to the space of zero mean value functions. However, a degeneration occurs when $r$ goes to infinity: up to normalization, it becomes $\Delta_0+2$, which is no longer invertible in this restricted space. \\

\indent From now on we assume $m>0$. Let $s$ be the function that gives the distance of a point
of $(M,g_m)$ to $\partial M$. Using $s\in[0,+\infty)$ as coordinate, one can write
\begin{equation}\label{gmcoords}
  g_m = \frac{dr^2}{1+r^2-\frac{2m}{r}}+r^2g_0 = ds^2 + \sinh^{2}(s)v_m(s)g_0,
\end{equation}
\noindent where $v_m$ is a positive function defined on $[0,+\infty)$ that has the following expansion as $s$ goes to infinity:
$$v_m(s)=1+ \frac{2m}{3\sinh^3 s} + O(\exp(-5s)).$$
\indent Although we have explicit formulas as in Proposition \ref{geomAdSS} only when we use the $r$ coordinate, it will be more convenient to use the $s$ coordinate. We will then consider $g_m$ to be defined on $M=[0,+\infty)\times S^2$ by formula (\ref{gmcoords}) above, and as a small abuse of notation we use $s$ both for the first coordinate of a point $p\in M$ and for the function $r\in (r_{0},+\infty) \mapsto s(r)\in (0,+\infty)$ that gives the coordinate change described above. It is worth noting that, as a function of $s$, the $r$ coordinate expands as $r=\sinh(s)(1+O(\exp(-3s)))$ as $s$ goes to infinity. In particular, for example, the mean curvature of the coordinate spheres $S_s:=S_{r(s)}$ behaves as 
\begin{equation*}
  H_{m}(s)=2\frac{\cosh s}{\sinh s} - \frac{2m}{\sinh^{3} s} + O(\exp(-5s)) \quad \text{as $s$ goes to infinity}.
\end{equation*}

\indent To conclude our description of the Schwarzschild anti-de Sitter space, it may be useful to remark that the double of $M$, $\hat{M}=(-\infty,+\infty)\times S^2$, can be endowed with a (smooth) Riemannian metric whose restriction to $M=[0,+\infty)\times S^2$ is the metric $g_m$ and such that the involution $i: (s,x) \in \hat{M} \mapsto (-s,x) \in\hat{M}$ is an isometry fixing $\{0\}\times S^2$, a surface identified with the totally geodesic boundary of $(M,g_m)$. We still denote this metric by $g_m$.

\medskip

\indent Now we define the class of metrics on $M=[0,+\infty)\times S^2$ we are going to work with.

\begin{definition}\label{Defperturbations}
Given $m>0$, let $\mathcal{M}(M,m)$ be the set of metrics $g$ of class $C^3$ on
$M=[0,+\infty)\times S^{2}$ such that
\begin{itemize}
  \item[$a)$] $\partial M=\{0\}\times S^2$ is a minimal surface in $(M,g)$; and
  \item[$b)$] There exists a constant $C>0$ such that for every $p=(s,x)\in M$,
  $$\left(\|g-g_m\| + \|\nabla^m g\| +\|(\nabla^m)^2 g\|+\|(\nabla^m)^3 g\|\right)(p) \leq C \exp(-4s).$$
  Here the norm is calculated with respect to the metric $g_m$ and $\nabla^m$ denotes the Levi-Civita connection of $g_m$.
\end{itemize}
\end{definition}

\indent The space $\mathcal{M}(M,m)$ has a distance function
\begin{equation*}
 d(g_1,g_2):= \sup_{p\in M}\left(\sum_{i=0}^{3}\exp(4s)\|(\nabla^{m})^i(g_1-g_2)\|_{g_m}(p)\right).
\end{equation*}
\indent We remark that each metric in $\mathcal{M}(M,m)$ is asymptotically hyperbolic with total mass $m$, according to the definitions of \cite{W} and \cite{CH}. The control of $g$ up to the third derivative is needed in order to apply the results of \cite{NT2}. \\
\indent Observe that we do not assume a priori that $\partial M$ is outermost.\\
\indent Given $g\in\mathcal{M}(M,m)$ one can calculate the expansions of its Ricci tensor, its scalar curvature and the mean curvature of the coordinate spheres in $(M,g)$ as follows: one adds terms of order $O(\exp(-4s))$ to the expansion of the corresponding quantities of $(M,g_m)$ in the $s$ coordinate. \\ 

\indent We finish this section by discussing the geometry of surfaces in $(M,g)$, $g\in\mathcal{M}(M,m)$. \\
\indent In this paper, we consider closed surfaces $\Sigma\subset M=[0,+\infty)\times S^{2}$ such that $M\setminus\Sigma$ has two connected components, one of them containing $\partial M$.
We define the inner radius and the outer radius of such $\Sigma$ to be
\begin{equation*}
  \underline{s} :=\min\{s(x);\,x\in\Sigma\} \quad \text{and} \quad \overline{s} :=\max\{s(x);\,x\in\Sigma\},
\end{equation*}
\noindent respectively. We also use the convention that the unit normal vector $N$ points toward the unbounded component of $M\setminus\Sigma$ and that the mean curvature is the trace of the second fundamental form $A$ given by $A(X,Y):=g(\nabla_{X} N, Y)$ for every pair of vectors $X,Y$ tangent to $\Sigma$. \\
\indent A constant mean curvature surface $\Sigma$ in $(M,g)$ is called weakly stable when its Jacobi operator,
\begin{equation*}
  L_{\Sigma} := \Delta_{\Sigma} + Ric(N,N) + |A|^{2},
\end{equation*}
\noindent is such that $-\int_{\Sigma} L_{\Sigma}(\phi)\phi d\Sigma \geq 0$ for all functions $\phi\in C^{\infty}(\Sigma)$ with $\int_{\Sigma}\phi d\Sigma = 0$. This analytic definition is equivalent to the geometric one that the second variation of the area of $\Sigma$ under volume preserving variations is non-negative. \\
\indent Throughout this paper, we will frequently consider surfaces in $(M,g)$ that are graphical over a coordinate sphere $S_{s}$. Given $u\in C^2(S^2)$, we write
 \begin{equation*}
  S_{s}(u):=\{(s+u(x),x)\in M;\, x\in S^2\}.
 \end{equation*}
\indent We denote by $H(s,u,g)$ the mean curvature of $S_{s}(u)$ in $(M,g)$, which we consider to be a real function defined on $S^2$. 

\begin{remark} \label{curvmedia}
 In the Section $3$, it will be important to understand how $H(s,u,g)$ depends on the variables $s,u$ and $g$. On way to analyze this is the following. Choose $(x^0=s,x^1,x^2)$ local coordinates on $M$, where $x=(x^1,x^2)$ are local coordinates on $S^2$. The local parametrization
$x \mapsto (x^0+u(x^1,x^2),x^1,x^2)$ of $S_{s}(u)$ gives explicit tangent vectors to $S_{s}(u)\subset (M,g)$, whose coordinates depends only on $\partial_{i}u(x)$ for $i=1,2$. From this, we conclude that the coordinates of the normal field $N$ to $S_{s}(u)$ in $(M,g)$ are smooth functions of $\partial_{i}u(x)$ and $g_{AB}(x^0+u(x^1,x^2),x^1,x^2)$ for all $i,j=1,2$, $A,B=0,1,2$. By the definition of the mean curvature, one then verify that $H(s,g,u)(x)$ is a smooth function of $\partial_{i}\partial_{j} u(x)$, $\partial_{i}u(x)$, $g_{AB}(x^0+u(x),x^1,x^2)$ and $\partial_{A}g_{BC}(x^0+u(x),x^1,x^2)$ for all $i,j=1,2$, $A,B,C=0,1,2$. From this rough analysis, for example, we see that taking derivatives of $H(s,u,g)$ with respect to the $s$ and $u$ variables requires the functions $g_{AB}$ to be twice differentiable.
\end{remark}


\section{Foliation of compact regions}\label{Secao3}

\indent For metrics $g\in\mathcal{M}(M,m)$ that are close enough to $g_m$, we use the implicit function theorem to construct a family of weakly stable CMC spheres on compact regions of $(M,g)$. \\
\indent For the rest of the paper, we fix some $\alpha\in (0,1)$.

\begin{theorem}\label{nocompacto}
  Let $m>0$. Given $S>0$, there exists $\eta_0>0$ such that for all $\eta\in (0,\eta_0)$ there exists $\epsilon>0$ with the following property: \\
  \indent For every $g\in\mathcal{M}(M,m)$ with $d(g,g_m)<\epsilon$ and for every $s\in[0,S]$ there exists a unique function $u(s,g)\in C^{2,\alpha}(S^2)$ with $\|u(s,g)\|_{C^{2,\alpha}} < \eta$ and $\int_{S^2} u(s,g) dS^2=0$ such that the surface 
  \begin{equation*}
   \Sigma_{s}(g):=S_s(u(s,g)) =\{ (s+u(s,g)(x), x)\in M;\, x\in S^2 \}
  \end{equation*}
  \noindent has constant mean curvature with respect to the metric $g$. \\
  \indent Moreover, for every $g\in\mathcal{M}(M,m)$ with $d(g,g_m)<\epsilon$, $\Sigma_{0}(g)=\partial M$ and the family $\{\Sigma_s(g)\}_{s\in[0,S]}$ gives a foliation of a compact region of $(M,g)$ by weakly stable CMC spheres, with positive mean curvature if $s\in(0,S]$. \\
  \indent Finally, when $S>s(3m)$, given any constant $\kappa>0$ and any compact interval $I\subset (s(3m),S]$, it is possible to choose the $\epsilon$ above in such way that for every $g\in\mathcal{M}(M,m)$ with $d(g,g_m)<\epsilon$ the  mean curvature $H_{g}(s)$ of $\Sigma_s(g)$ in $(M,g)$ is monotone decreasing on $I$ and satisfies $|H_{g}(s)-H_{m}(s)| < \kappa$ for every $s\in I$.
\end{theorem}

\begin{proof}
 \indent For a minor technical reason, choose $\theta > 0$. In order to consider the geometry of graphs of small functions over $\partial M$, we have to extend each metric in $\mathcal{M}(M,m)$ to a fixed neighborhood of $M$ in $\hat{M}=(-\infty,+\infty)\times S^2$. It is always possible to extend $g\in\mathcal{M}(M,m)$ as a $C^3$ metric $\tilde{g}$ on $\tilde{M}:=[-\theta,+\infty)\times S^2$ in such way that $||\tilde{g}-g_m||_{C^{3}(\tilde{M})}\leq\lambda||g-g_m||_{C^{3}(M)}$ for some fixed $\lambda > 1$. In particular, if $d(g,g_m)<\epsilon$ then $||\tilde{g}-g_m||_{C^{3}(\tilde{M})}<\lambda\epsilon$. For simplicity we refer to this extension of $g$ using the same letter. \\
   \indent On the round unit sphere $(S^2,g_0)$, consider the Banach spaces
  \begin{equation*}
   E=\{u\in C^{2,\alpha}(S^2); \int_{S^2} u dS^2=0\} \,\, \text{and} \,\, F=\{u\in C^{0,\alpha}(S^2); \int_{S^2} u dS^2=0\},
  \end{equation*}
  \noindent endowed with their standard norms. \\
  \indent Given $\eta_0\in(0,\theta)$, for every $s\in[0,S]$ and every $u\in E$ with $||u||_{C^{2,\alpha}}<\eta_0$ we can consider the surfaces
  \begin{equation*}
   S_s(u)=\{(s+u(x),x)\in \tilde{M};\,x\in S^2\}.
  \end{equation*}
  \indent Let $H(s,u,g)$ be the mean curvature of $S_{s}(u)$ with respect to a metric $g\in \mathcal{M}(M,m)$. It defines a $C^1$ map from $[0,S]\times (B(0,\eta_0)\subset E) \times \mathcal{M}(M,m)$ to $C^{0,\alpha}(M)$, see Remark \ref{curvmedia} in the end of Section \ref{Secao2}. \\
  \indent Let $\Phi: [0,S]\times \mathcal{M}(M,m)\times(B(0,\eta_0)\subset E) \rightarrow F$ be the $C^1$ map given by
  \begin{equation*}
    \Phi(s,g,u)=  H(s,u,g)-\frac{1}{4\pi}\int_{S^2}{H(s,u,g)dS^2}.
  \end{equation*}
  \indent Observe that $\Phi(s,g,u)=0$ if and only if $S_{s}(u)$ is a CMC surface in $(\tilde{M},g)$. In particular, $\Phi(s,g_{m},0)=0$ for all $s\in[0,S]$. \\
  \indent We claim that, for every $s\in [0,S]$, $D\Phi_{(s,g_m,0)}$ is an isomorphism when restricted to $E$. In fact, for every $v\in E$, the family $t \mapsto S_s(tv)$ is a normal variation of the coordinate sphere $S_{s}$ in $(M,g_m)$ with speed $v$. Therefore, if $L_s$ is the Jacobi operator of $S_s$ with respect to $g_m$, we have
  \begin{equation*}
    \frac{d}{dt}_{|_{t=0}}\Phi(s,g_m,tv) = -L_{s}(v)+\frac{1}{4\pi}\int_{S^2} L_{s}(v)dS^2 = -L_{s}(v).
  \end{equation*}
  \indent The last equality follows because $L_s(v)=(1/r^2)(\Delta_0+(2-6m/r))(v)$ (see Proposition \ref{geomAdSS}, considering the coordinate change $r=r(s)$) and $v$ has zero mean value. Since $m>0$, $\Delta_0+(2-6m/r)$ is an invertible operator from $E$ to $F$ for all $s\in[0,S]$ and the claim follows. \\ 
  \indent Therefore we can apply the implicit function theorem: there exists (a possibly smaller) $\eta_0>0$ such that for all $\eta\in(0,\eta_0)$, there exists a small ball $B$ around $(s,g_m)$ in $[0,S]\times\mathcal{M}(M,m)$ and a $C^1$ function $(\tilde{s},g)\in B \mapsto u(\tilde{s},g)\in B(0,\eta)\subset E$ such that $u(s,g_m)=0$ and $u(\tilde{s},g)$ is uniquely defined in $B(0,\eta)$ by the equation $\Phi(\tilde{s},g,u(\tilde{s},g))=0$ for all $(\tilde{s},g)\in B$. \\
  \indent Since $u$ is a $C^1$ function of $s$ and $g$, by compactness we can choose $\eta_0$ in such way that for all $\eta\in(0,\eta_0)$ there exists $\epsilon>0$ such that $u$ is uniquely defined on $[0,S]\times\{g\in\mathcal{M}(M,m);d(g,g_m)<\epsilon\}$ and takes values on $B(0,\eta)\subset E$. Thus for each metric $g\in\mathcal{M}(M,m)$ with $d(g,g_m)<\epsilon$ we have constructed a family 
  \begin{equation*}
   \{\Sigma_{s}(g)\}_{s\in[0,S]}:=\{S_s(u(s,g))\}_{s\in[0,S]}
  \end{equation*}
  \noindent of CMC spheres in $(\tilde{M},g)$, where $u(s,g)\in E$ has norm $\|u(s,g)\|_{C^{2,\alpha}} < \eta$. Notice that the surfaces $\Sigma_{s}(g)$ depend smoothly on $s$ and $g$, and also that $\{\Sigma_{s}(g_m)\}$ is precisely the foliation of $M$ by coordinate spheres. \\
  \indent Since $\partial M=S_{0}(0)$ is minimal for all metrics $g\in\mathcal{M}(M,m)$, the uniqueness  of the function $u$ above constructed implies that $u(0,g)=0$, i.e., $\Sigma_{0}(g)=\partial M$ for every $g\in\mathcal{M}(M,m)$ with $d(g,g_m)<\epsilon$.\\
  \indent In order to prove that $\{\Sigma_{s}(g)\}$ is a foliation of some region of $M$, we have to analyze the sign of its lapse function, that is, its normal speed. Since for $g_m$ the constructed family $\{\Sigma_{s}(g_m)\}$ is a foliation, its lapse function is positive on $[0,S]$, hence the lapse function of $\{\Sigma_{s}(g)\}_{s\in[0,S]}$ with respect to all $g\in\mathcal{M}(M,m)$ with $d(g,g_m)<\epsilon$ is also positive, at least when we choose a possibly smaller $\epsilon$. Since each of these families starts at $\partial M=\{0\}\times S^2\subset\tilde{M}$, the families $\{\Sigma_{s}(g)\}_{s\in[0,S]}$ foliate a compact region of $M$.  \\
  \indent To see that $\Sigma_{g}(s)$ has positive mean curvature in $(M,g)$ for all $s\in(0,S]$, observe that this is true for $\Sigma_{s}(g_m)$ in $(M,g_m)$ and also that, by Proposition \ref{geomAdSS}, $H_{m}'(0)=-(1/r_{0}^2)(2-6m/r_0)>0$. Moreover, $H_{g}(s):=H(s,u(s,g),g)$ is a $C^1$ function of $s$ and $g$. Therefore by continuity we can arrange $\epsilon$ in such way that for every $g\in\mathcal{M}(M,m)$ with $d(g,g_m)<\epsilon$ the surface $\Sigma_{s}(g)$ has positive mean curvature in $(M,g)$ for all $s\in(0,S]$. \\
  \indent Now we argue that the leaves are weakly stable. In fact, since $m>0$, for every $s\in[0,S]$ the Jacobi operator $L_s$ of $\Sigma_{s}(g_m)$ in $(M,g_m)$ satisfies
  \begin{equation*}
    -\int_{\Sigma_{s}(g_m)} L_{s}(\phi)\phi d\Sigma_s(g_m) = \int_{S^2} |\nabla_{0}\phi|^2-\left(2-\frac{6m}{r(s)}\right)\phi^2 dS^2 \geq \frac{6m}{r(S)}\int_{S^2} \phi^2 dS^2 
   \end{equation*}
  \noindent for every $\phi\in C^{\infty}(\Sigma_s(g_m))$ with $\int \phi d\Sigma_s(g_m) =0$. Arguing by contradiction we then conclude that for a possibly smaller $\epsilon$ there exists a constant $c>0$ such that for every $g\in\mathcal{M}(M,m)$ with $d(g,g_m)<\epsilon$ and every $s\in[0,S]$ the Jacobi operator $L_{(s,g)}$ of the surface $\Sigma_{s}(g)$ in $(M,g)$ is such that 
  \begin{equation*}
    -\int_{\Sigma_{s}(g)} L_{(s,g)}(\phi)\phi d\Sigma_s(g) \geq c\int_{\Sigma_{s}(g)} \phi^2 d\Sigma_s(g)\geq 0  
  \end{equation*}
  \noindent for every $\phi\in C^{\infty}(\Sigma_s(g))$ with $\int \phi d\Sigma_s(g) =0$, i.e., $\Sigma_s(g)$ is weakly stable. \\
  \indent The last statement of the theorem also follows by continuity, since $H_{m}'(s)$ $<0$ on $(s(3m),+\infty)$.
\end{proof}

\medskip

\indent Notice that in the previous result the mean curvature of the graph $\Sigma_{s}(g)=S_{s}(u(s,g))$
is not prescribed and can be different from $H_{m}(s)$. Having in mind the matching argument of Section \ref{Secao6}, we finish this section showing the existence and uniqueness of a graph over the coordinate sphere $S_{s}$ with prescribed mean curvature $H_{m}(s)$ in $(M,g)$, at least when $d(g,g_m)$ is sufficiently small and $s$ is large enough. More precisely, we have:

\begin{theorem}\label{nocompacto2}
  Let $m>0$. Given a compact interval $I\subset (s(3m),+\infty)$, there exists $\eta_0>0$ such that for all $\eta\in (0,\eta_0)$ there exists $\epsilon>0$ with the following property: \\
  \indent For every $g\in\mathcal{M}(M,m)$ with $d(g,g_m)<\epsilon$ and for every $s\in I$ there exists a unique function $h(s,g)\in C^{2,\alpha}(S^2)$ with $\|h(s,g)\|_{C^{2,\alpha}} < \eta$ such that the surface
  \begin{equation*}
   S_s(h(s,g)) =\{ (s+h(s,g)(x), x)\in M;\, x\in S^2 \}
  \end{equation*}
  \noindent has constant mean curvature $H_m(s)$ in $(M,g)$. \\
  \indent Moreover, the family $\{S_s(h(s,g))\}_{s\in I}$ gives a foliation of a compact region of $(M,g)$ by weakly stable CMC spheres.
  \end{theorem}

\begin{proof}
  Following the notations of Theorem \ref{nocompacto}, given $\eta_0>0$ sufficiently small we consider
  the $C^1$ map $\Psi: I \times \mathcal{M}(M,m)\times(B(0,\eta_0)\subset C^{2,\alpha}(S^2)) \rightarrow C^{0,\alpha}(S^2)$ given by $\Psi(s,g,u)=  H(s,u,g)-H_{m}(s)$. Notice that $\Psi(s,g_m,0)=0$ for all $s\in I$. \\
  \indent The linearisation of $\Psi$ at $(s,g_m,0)$ is such that, for every $v\in C^{2,\alpha}(S^2)$,
  \begin{equation*}
    D\Phi_{(s,g_m,0)}(0,0,v) = -L_{s}(v) = -\frac{1}{r^2}(\Delta_{0} + (2 - \frac{6m}{r})),
  \end{equation*}
  \noindent where we use the coordinate $r=r(s)$, see Proposition \ref{geomAdSS}. Since $I\subset (s(3m),+\infty)$, the map $v\in C^{2,\alpha}(S^2) \mapsto L_{s}(v) \in C^{0,\alpha}(S^2)$ is an isomorphism for all $s\in I$. Hence, we can apply the implicit function theorem and the result follows by the same arguments presented in Theorem \ref{nocompacto}.
\end{proof}

\section{Foliation near the infinity}\label{Secao4}

\indent The next theorem is the version of the existence and uniqueness theorem of A. Neves and G. Tian \cite{NT2} adapted to the asymptotically hyperbolic manifolds $(M,g)$ where $g$ belongs to the space of metrics $\mathcal{M}(M,m)$ (we refer the reader to Theorem 2.2 and the proof of Theorem 8.2 in \cite{NT2}).
\begin{theorem}[\cite{NT2}]\label{noinfinito}
  Let $m>0$. Given $\epsilon_0>0$, there exists $\delta>0$, $C>0$ and $\underline{s}_{0}>s(3m)$
  with the following properties:
  \begin{itemize}
    \item[$1)$] Given $g\in\mathcal{M}(M,m)$ with $d(g,g_m)<\epsilon_0$,
    for all $\ell\in(2,2+\delta)$, there exists a unique sphere $\Sigma_{\ell}=\Sigma_{\ell}(g) \subset M$ 
    such that
    \begin{itemize}
      \item[$a)$] $M\setminus\Sigma_{\ell}$ has two connected components, one of them containing $\partial M$;
      \item[$b)$] $\Sigma_{\ell}$ is a weakly stable constant mean curvature sphere in $(M,g)$
      with mean curvature $H=\ell$; and
      \item[$c)$] The inner radius $\underline{s}_{\ell}$ and the outer radius $\overline{s}_{\ell}$ of $\Sigma_{\ell}$ satisfy
    \begin{eqnarray*}
      \underline{s}_{\ell} \geq \underline{s}_{0} & \quad \text{and} & \overline{s}_{\ell} - \underline{s}_{\ell} \leq 1.
    \end{eqnarray*}
    \end{itemize}
    \item[$2)$] The family $\{\Sigma_{\ell}\}_{\ell\in(2,2+\delta)}$ gives a smooth foliation of 
    the complement of a compact set of $M$ and $\lim_{\ell\rightarrow 2} \underline{s}_{\ell} = +\infty.$ 
    \item[$3)$] Given $g\in\mathcal{M}(M,m)$ with $d(g,g_m)<\epsilon_0$ and $\ell\in(2,2+\delta)$,
    if for the above surface $\Sigma_{\ell}$ in $(M,g)$ we define $\hat{s}_{\ell}$ by the equality 
    $$|\Sigma_{\ell}|=4\pi\sinh^{2}\hat{s}_{\ell},$$
    \noindent then:
    \begin{itemize}
      \item[$a)$] If we set $w_{\ell}(p)=s(p)-\hat{s}_{\ell}$ for $p\in \Sigma_l$, then
      \begin{eqnarray*}
        \sup_{\Sigma_{\ell}}|w_{\ell}| \leq C\exp(-\underline{s}_{\ell}) & \text{and} & \int_{\Sigma_{\ell}}|\partial_{s}^{\top}|^2 d\Sigma_{\ell} \leq C\exp(-2\underline{s}_{\ell}).
      \end{eqnarray*}
      \item[$b)$]
      \begin{equation*}
        \int_{\Sigma_{\ell}} |\mathring{A}_{\ell}|^2 d\Sigma_{\ell} \leq C\exp(-4\underline{s}_{\ell}).
      \end{equation*}
      \item[$c)$] There exists a function $f\in C^{2}(S^{2})$ with 
      $\|f\|_{C^{2}}\leq C$ such that 
      \begin{equation*}
        \Sigma_{\ell}= S_{\hat{s}_{\ell}}(f)=\{(\hat{s}_{\ell}+f(x),x)\in M;\, x\in S^2\}.         
      \end{equation*}
    \end{itemize}
  \end{itemize}
\end{theorem}

\begin{remark}
  Using the terminology of \cite{NT2}, the metrics $g\in\mathcal{M}(M,m)$ satisfy condition $(H)$ with same constants $r_1, C_2$ and $C_3$, and the constant $C_1$ depends only on $d(g,g_m)$. Therefore, if we fix the constant $C_4=1$, all constants appearing in Neves and Tian's theorem depend only on the distance between the metric $g\in\mathcal{M}(M,m)$ and $g_m$. This uniform dependence on $d(g,g_m)$ is crucial to the matching argument.
\end{remark}  

\begin{remark} \label{crucial}
  Theorem \ref{noinfinito} is proved by the continuity method (Theorem 8.2 in \cite{NT2}). A fundamental step is to prove that a weakly stable CMC sphere in $(M,g)$ with mean curvature $\ell$ close enough to $2$ and large enough inner radius is graphical over \textit{some} coordinate sphere (this is the content of item $3)$ of the above theorem) and that its Jacobi operator is invertible (see Proposition 8.1 in \cite{NT2}). To prove the \textit{existence} of the foliation, following the argument of \cite{NT2}, page $92$, given $g\in\mathcal{M}(M,g)$ with $d(g,g_m)<\epsilon_0$, consider the family of metrics $g_t=(1-t)g_m+tg \in \mathcal{M}(M,g)$, $t\in [0,1]$. The rough idea is then to prove that for each $s\in(s(3m),+\infty)$ large enough the set $A$ of $t\in[0,1]$ such that there exists a weakly stable CMC sphere $\Sigma\subset (M,g_t)$ with mean curvature $H_m(s)$ and large enough inner radius is open and closed. \\
\indent The set $A$ clearly contains $t=0$. The proof that $A$ is open follows the reasoning of the proof of Theorem $3$: given $t$ in this set, since the Jacobi operator of $\Sigma$ in $(M,g)$ is invertible and since $\Sigma$ is a graph over a coordinate sphere, the implicit function theorem implies the existence of small graphs over the \textit{same} coordinate sphere which have the \textit{same} mean curvature with respect to nearby $g_{t'}$ metrics. The proof that the set $A$ is closed uses the uniform estimates on the norm of the function that gives the graph (see item $3)$ of the above theorem).\\
\indent Since the CMC sphere in $(M,g_m)$ with mean curvature $H_m(s)$ is precisely the coordinate sphere $S_s$, we conclude the following. Given $\epsilon_0>0$, let $\delta$, $\overline{s}_0$ and $C$ be given by Theorem \ref{noinfinito}. Let $I\subset (\overline{s}_0,+\infty)$ be a compact interval such that $H_{m}(s)\in(2,2+\delta)$ for all $s\in I$. Then, the corresponding $\eta_0>0$ given by Theorem \ref{nocompacto2} is such that for every $\eta\in(0,\eta_0)$ there exists $\epsilon\in(0,\epsilon_0)$ such that for every $g\in\mathcal{M}(M,m)$ with $d(g,g_m)<\epsilon$ the surfaces of the foliation $\{\Sigma_{\ell}(g)\}$ described in Theorem \ref{noinfinito} with mean curvature $\ell=H_m(s)\in(2,2+\delta)$ for $s\in I$ are exactly the surfaces obtained in Theorem \ref{nocompacto2}. \\
\indent Observe that the above argument does not make any reference to the \textit{uniqueness} statement of Theorem \ref{noinfinito}. And since it contains all the information needed for the matching argument, we will use only it in the proof of Theorem \ref{colagem} (see Section \ref{Secao6}).
\end{remark}


\section{Limit of the Hawking mass}\label{Secao5}

\indent Let $g$ be a metric in $\mathcal{M}(M,m)$ with scalar curvature $R\geq -6$. Recall that the Hawking mass of a closed surface $\Sigma$ in $(M,g)$ is 
\begin{equation*}
  m_{H}(\Sigma)=\sqrt{\frac{|\Sigma|}{16\pi}}\left( 1-\frac{1}{16\pi}\int_{\Sigma}( H^2 - 4 )d\Sigma \right).
\end{equation*}
\indent We want to calculate the limit of the right hand side of the above expression for the weakly stable CMC spheres $\Sigma_{\ell}$ in $(M,g)$ given by Theorem \ref{noinfinito} as they approach the infinity. In order to do this, we need the following consequence of Gauss equation (see \cite{NT1}).

\begin{lemma} \label{apriori}
  Given $g\in\mathcal{M}(M,m)$, let $\{\Sigma_{t}\}_{t>t_0}$ be a family of constant mean curvature spheres in $(M,g)$ such that $\underline{s}_{t}\rightarrow +\infty$ as $t$ goes to infinity. Then
  \begin{equation*}
    (H_{t}^{2}-4)|\Sigma_{t}| = 16\pi - \int_{\Sigma_{t}} \frac{8m-12m|\partial_{s}^{\top}|^2}{\sinh^{3}s} d\Sigma_{t} + 2\int_{\Sigma_{t}}|\mathring{A}_t|^2d\Sigma_{t} + |\Sigma_{t}|O(\exp(-4\underline{s}_{t})).
  \end{equation*}
\end{lemma}

\begin{proof}
  \indent Let $K_{t}$ be the Gaussian curvature of $\Sigma_{t}$. The Gauss equation for $\Sigma_{t}$ in $(M,g)$ can be written as 
  \begin{align}\label{eqGauss}
             2K_t & = R-2Ric(N_t,N_t)+H_t^2-|A_t|^2 \\
    \nonumber     & = (R+6)-2(Ric(N_t,N_t) + 2) + \frac{H_{t}^{2}-4}{2} - |\mathring{A}_{t}|^2.
  \end{align}
  \indent By Proposition \ref{geomAdSS}, for metrics $g\in\mathcal{M}(M,m)$, if $\{\partial_s,e_1,e_2\}$ is a $g_m$-orthonormal referential, then
  \begin{align*}
    Ric(\partial_{s},\partial_{s})  & =  -2 -\frac{2m}{\sinh^{3}s}+ O(\exp(-4s)),  \\
    Ric(e_{i},e_{j})                & =  \left(-2+\frac{m}{\sinh^{3}s}\right)\delta_{ij} + O(\exp(-4s)),  \\
    Ric(\partial_{s},e_{i})         & =  O(\exp(-4s)), \quad \text{and} \\
    R+6                             & =  O(\exp(-4s)). 
  \end{align*}
  \indent Considering the $g_m$-orthogonal decomposition $N_{t}=a\partial_{s}+X$, it follows that
  \begin{equation}\label{degauss}
    4K_{t} = (H_{t}^{2}-4) + \left(\frac{8m - 12m |X|^2_{g_m}}{\sinh^{3}s}\right) - 2|\mathring{A}_t|^2 + O(\exp(-4s)).
  \end{equation}
  \indent Observe that if $\nu$ is the unit normal of a coordinate sphere in $(M,g)$, then $\partial_{s}=\nu+W$, where $g(\nu,W)=O(\exp(-4s))$ and $|W|_g=O(\exp(-4s))$. Hence, the $g$-orthogonal decomposition $N_{t}=b\nu+Y$ is such that $|X|^2_{g_m}=|Y|^2_g+O(\exp(-4s))$. On the other hand, if $\nu=bN_{t}+\nu^{\top}$ is the $g$-orthonormal decomposition of $\nu$ corresponding to the tangent space of $\Sigma_t$ and its $g$-normal $N_t$, we have $|Y|^2_{g}=|\nu^{\top}|^2_{g}$. Therefore one can change $|X|_{g_m}^2$ by $|\partial_{s}^{\top}|_{g}^2$ in (\ref{degauss}). \\
  \indent Since $\Sigma_{t}$ is a sphere of constant mean curvature, the lemma follows after integration of (\ref{degauss}).
\end{proof}

\begin{proposition}\label{limitemassa}
Let $m>0$.  Given $g\in\mathcal{M}(M,m)$, the family $\{\Sigma_{\ell}\}_{\ell\in(2,2+\delta)}$ in $(M,g)$ given by Theorem \ref{noinfinito} is such that
  \begin{equation*}
    \lim_{\underline{s}_{\ell}\rightarrow +\infty} \sqrt{\frac{|\Sigma_{\ell}|}{16\pi}}\left( 1-\frac{1}{16\pi}\int_{\Sigma_{\ell}}( H_{\ell}^2-4)d\Sigma_{\ell} \right) = m.
  \end{equation*}
\end{proposition}

\begin{proof}
  \indent Let $\hat{s}_{\ell}$ be defined by $|\Sigma_{\ell}| = 4\pi\sinh^{2}\hat{s}_{\ell}$. We use the information given by Theorem \ref{noinfinito}, item $3)$, and calculate all expansions as $\underline{s}_{\ell}$ goes to infinity. \\
  \indent Lemma \ref{apriori} and the estimates of Theorem \ref{noinfinito}, item $3)$ on the behavior of $|\partial_{s}^{\top}|$ and $|\mathring{A}_{\ell}|$ implies
  \begin{multline}\label{primeira}
    \sqrt{\frac{|\Sigma_{\ell}|}{16\pi}}\left( 1-\frac{1}{16\pi}\int_{\Sigma_{\ell}}( H_{\ell}^2-4)d\Sigma_{\ell} \right) = \\ = \frac{|\Sigma_{\ell}|^{1/2}}{8\pi^{3/2}} \left( \int_{\Sigma_{\ell}} \frac{m}{\sinh^{3}s}d\Sigma_{\ell} + |\Sigma_{\ell}|O(\exp(-4\underline{s}_{\ell}))+O(\exp(-4\underline{s}_{\ell})) \right).
  \end{multline}
  \indent In order to analyze (\ref{primeira}), recall that $\Sigma_{\ell}$ is a graph over the coordinate sphere $S_{\hat{s}_{\ell}}$ given by a function $f$ whose $C^2$ norm is bounded for all $\ell\in (2,2+\delta)$ by the same constant $C>0$. Therefore $|\underline{s}_{\ell}-\hat{s}_{\ell}| = |\min_{x\in S^2}\{f(x)\}| \leq C $ for all $\ell$, which implies
  \begin{equation}\label{segunda}
    |\Sigma_{\ell}|=4\pi\sinh^{2}\hat{s}_{\ell}=O(\exp(2\underline{s}_{\ell})).
  \end{equation}
  \indent On the other hand, for every $p=(s,x)\in \Sigma_{\ell}$, since $|w_{\ell}(p)| \leq C\exp(-\underline{s}_{\ell})$,
  \begin{equation*}
    \frac{\sinh\hat{s}_{\ell}}{\sinh s}=\frac{\sinh(s-w_{\ell}(p))}{\sinh s}=\cosh w_{\ell}(p)-\frac{\cosh{s}}{\sinh{s}}\sinh{w_{\ell}(p)}=1 + O(\exp(-\underline{s}_{\ell})).  
  \end{equation*}
  \indent Therefore
  \begin{equation}\label{terceira}
    \int_{\Sigma_{\ell}} \frac{|\Sigma_{\ell}|^{1/2}}{\sinh^{3}(s)} d\Sigma_{\ell} = \frac{(4\pi)^{3/2}}{|\Sigma_{\ell}|}\int_{\Sigma_{\ell}} \left(\frac{\sinh\hat{s}}{\sinh s}\right)^{3} d\Sigma_{\ell} =8\pi^{3/2}(1+O(\exp(-\underline{s}_{\ell}))).
  \end{equation}
  \indent The proposition follows from (\ref{primeira}), (\ref{segunda}) and (\ref{terceira}).
\end{proof}

\medskip


\section{The global foliation}\label{Secao6}

\indent We now argue that when a metric $g\in\mathcal{M}(M,m)$ is sufficiently close to $g_m$ the foliations of $(M,g)$ constructed in Sections \ref{Secao3} and \ref{Secao4} match on their overlap. It is possible to argue using the uniqueness of Neves and Tian's foliation, but we follow a different reasoning using the implicit function theorem, see Remark \ref{crucial} in Section \ref{Secao4}.

\begin{theorem}\label{colagem}
  Let $m>0$. There exists $\epsilon>0$ with the following property: \\
  \indent If $g\in\mathcal{M}(M,m)$ is such that $d(g,g_m)<\epsilon$, then there exists a foliation $\{\Sigma_t\}_{t\in[0,+\infty)}$ of $M$ such that:
  \begin{itemize}
   \item[$i)$] Each $\Sigma_{t}$ is a weakly stable CMC sphere in $(M,g)$, with positive mean curvature when $t>0$;
   \item[$ii)$] $\Sigma_{0}=\partial M$ is an outermost minimal surface in $(M,g)$; and
   \item[$iii)$] $\lim_{t\rightarrow +\infty} \sqrt{\frac{|\Sigma_t|}{16\pi}}\left( 1-\frac{1}{16\pi}\int_{\Sigma_t}( H_t^2-4)d\Sigma_t \right)=m$.
  \end{itemize}
\end{theorem}

\begin{proof}
  \indent Given an arbitrary $\epsilon_0>0$, let $\delta>0$, $C>0$ and $\underline{s}_{0}>s(3m)$ be given by Theorem \ref{noinfinito}. Recall that $H_m(s)$ is a smooth monotone decreasing function on the interval $(s(3m),+\infty)$ and converges to $2$ as $s$ goes to infinity. Let $s_{0}>\underline{s}_{0}$ be such that $H_m(s_{0})<2+\delta$. \\
  \indent Let $S > s_{0}+1$. Given this choice of $S$, we can choose $\eta<1$ and $\epsilon\in(0,\epsilon_0)$ sufficiently small in such way that Theorem \ref{nocompacto} holds for $S$ and Theorem \ref{nocompacto2} holds for the interval $[S-1,S]$ for every metric $g\in\mathcal{M}(M,m)$ with \mbox{$d(g,g_m)<\epsilon$}. \\
  \indent In particular, we can assume that for every metric $g\in\mathcal{M}(M,m)$ with $d(g,g_m)<\epsilon$ and for every $s\in[S-1,S]$ the surface $\Sigma_{\ell}=\Sigma_{\ell}(g)$ in $(M,g)$ described in Theorem \ref{noinfinito} with $\ell=H_{m}(s)$ is given by 
  \begin{equation*}
  \Sigma_{\ell} = S_{s}(h),
  \end{equation*}
  \noindent where $h$ is the unique function in $C^{2,\alpha}(S^2)$ with norm $<\eta$ such that its graph over $S_{s}$ has constant mean curvature $H_{m}(s)$, see the remarks after Theorem \ref{noinfinito}. \\
  \indent Given $\kappa\in(0,\eta)$, let $[c,d]$ be the image of the interval $[S-3\kappa/4,S-\kappa/4]$  under the map $H_{m}(s)$. We can moreover assume that $\epsilon\in(0,\epsilon_0)$ is sufficiently small  in such way that for every $g\in\mathcal{M}(M,m)$ with $d(g,g_m)<\epsilon$ the foliation $\{\Sigma^{1}_s(g)\}_{s\in[0,S]}$ constructed in Theorem \ref{nocompacto} has in particular the following properties:
  \begin{itemize}
   \item[$a)$] For every $s\in[S-1,S]$, there exists a function $u(s,g)\in C^{2,\alpha}(S^2)$ with 
   $\|u(s,g)\|_{C^{2,\alpha}} < \eta/2$
   such that 
   $$ \Sigma^{1}_s(g) = S_{s}(u(s,g)).$$
   \item[$b)$] The mean curvature $H_{g}(s)$ of $\Sigma^{1}_{s}(g)$ in $(M,g)$ 
   is a decreasing function on the interval $[S-1,S]$ with 
    \begin{equation*}\label{proximidade}
      |H_{g}(s)-H_{m}(s)|< (d-c)/4
    \end{equation*}
    for all $s\in [S-1,S]$. Hence, there exists an interval $(a,b)\subset[S-3\kappa/4,S-\kappa/4]$ such that for every $s\in (a,b)$ there exists a unique $\tilde{s}\in(S-3\kappa/4,S-\kappa/4)$ with $H_{g}(s)=H_{m}(\tilde{s})$. This defines $\tilde{s}$ as a function of $s$, which is of class $C^1$ since $H_{g}(s)$ is a $C^1$ function of $s$.
  \end{itemize}
  
  \indent \\
  \indent We now prove that the theorem is true for this choice of $\epsilon$. \\
  \indent In fact, fix some $g\in\mathcal{M}(M,m)$ with $d(g,g_m)<\epsilon$. By item $a)$ and $b)$ above, given $s\in (a,b)$, if we define the function $\tilde{h} = s - \tilde{s}(s) + u(s,g) \in C^{2,\alpha}(S^2)$, then 
  \begin{equation*}
    \|\tilde{h}\|_{C^{2,\alpha}} \leq |s-\tilde{s}(s)| + \|u\|_{C^{2,\alpha}} < \kappa/2 + \eta/2 < \eta
  \end{equation*}
  \noindent and the graph
  \begin{equation*}
    S_{\tilde{s}(s)}(\tilde{h}) = S_{s}(u(s,g)) = \Sigma^1_{s}(g)
  \end{equation*}
  \noindent has constant mean curvature $H_{g}(s)=H_{m}(\tilde{s}(s))$. These are the conditions that uniquely characterize the function that gives $\Sigma_{\ell}$ in $(M,g)$ with $\ell=H_{m}(\tilde{s}(s))$ as a graph over $S_{\tilde{s}(s)}$. Therefore $\Sigma_{s}^1(g) = \Sigma_{\ell}(g)$ where $\ell=H_m(\tilde{s}(s))$  for all $s\in (a,b)$.  \\
  \indent This proves that the foliations $\{\Sigma^{1}_{s}(g)\}$ and $\{\Sigma_{\ell}(g)\}$ given by Theorems \ref{nocompacto} and \ref{noinfinito}, respectively, match on their overlap. The obtained   foliation, which we reparametrize as $\{\Sigma_{t}\}_{t\in[0,+\infty)}$, is a foliation of $(M,g)$ by weakly stable CMC spheres, starting at the minimal $\Sigma_{0}=\partial M$, such that each $\Sigma_{t}$ has positive mean curvature for $t>0$ and such that
  \begin{equation*}
   \lim_{t\rightarrow +\infty} \sqrt{\frac{|\Sigma_t|}{16\pi}}\left( 1-\frac{1}{16\pi}\int_{\Sigma_t}( H_t^2-4)d\Sigma_t \right) = m,
  \end{equation*}
  \noindent see Theorem \ref{nocompacto}, Theorem \ref{noinfinito} and Proposition \ref{limitemassa}. It remains only to prove that $\partial M$ is an outermost minimal surface in $(M,g)$. This is a consequence of the Maximum Principle, since we proved $M\setminus\partial M$ is foliated by surfaces with positive mean curvature.
\end{proof}


\section{The Penrose inequality for perturbations of the Schwarzschild anti-de Sitter spaces of positive mass}\label{Secao7}

\indent Using the foliation by weakly stable CMC spheres constructed above on $(M,g)$, $g\in \mathcal{M}(M,m)$ sufficiently close to $g_m$, and the remark of H. Bray that the Hawking mass is monotone along such families (see \cite{B1}), we prove the Penrose inequality for this class of asymptotically hyperbolic manifolds.
\begin{theorem}\label{penrose}
  Given $m>0$, let $\epsilon>0$ be given by Theorem \ref{colagem}. If $g\in\mathcal{M}(M,m)$ with $d(g,g_m)<\epsilon$ has scalar curvature $R\geq -6$, then
  \begin{equation}\label{penroseineq}
    \left(\frac{|\partial M|}{16\pi}\right)^{\frac{1}{2}}+4\left(\frac{|\partial M|}{16\pi}\right)^{\frac{3}{2}} \leq m.
  \end{equation}
  \indent Moreover, equality holds if and only if $(M,g)$ is isometric to $(M,g_m)$.
\end{theorem}
\begin{proof} Let $\{\Sigma_t\}_{t\geq 0}$ be the foliation of $(M,g)$ by weakly stable CMC spheres constructed in Theorem \ref{colagem}. \\
\indent By assumption, $g$ has scalar curvature $R \geq -6$. Then the Hawking mass of $\Sigma_t$,
\begin{equation*}
  m_{H}(\Sigma_t)=\sqrt{\frac{|\Sigma_t|}{16\pi}}\left( 1-\frac{1}{16\pi}\int_{\Sigma_t}(H_{t}^{2} - 4)d\Sigma_t \right), 
\end{equation*}
\noindent is monotone non-decreasing in $t$. In fact, we can be more precise: \\

\noindent \textbf{Claim:} $m_{H}^{'}(\Sigma_{t}) \geq 0$. Moreover, $m_{H}^{'}(\Sigma_{t})$ is zero
for $t>0$ if and only if $\Sigma_{t}$ satisfies the following properties:
\begin{itemize}
 \item[$a)$] $R$ is constant and equal to $-6$ along $\Sigma_{t}$;
 \item[$b)$] $\Sigma_{t}$ is totally umbilic; and
 \item[$c)$] $\Sigma_{t}$ has constant Gaussian curvature.
\end{itemize}

\medskip

\indent The proof of the claim goes as follows. Choose a parame\-tri\-zation of this foliation
by some function $F:[0,+\infty)\times S^2 \rightarrow M$ such that for each $t\in[0,+\infty)$, 
$F_t: S^2 \rightarrow M$ is a parametrization of $\Sigma_t$ and $\partial_t F$ does not vanish.
Let $\rho_t$ be the lapse function, i.e., $\rho_t= g(N_t,\partial_t F)$ where $N_t$ is the unit
normal of the surface $\Sigma_t$ pointing to infinity. Since we have a foliation, $\rho_t > 0$ on $\Sigma_t$ for all $t$. We also define the mean value $\overline{\rho}_t=\int \rho_t d\Sigma_t/|\Sigma_t|$. \\
\indent By the first variation formula of the Hawking mass (see, for example, \cite{MN}), we have  
\begin{equation*}
  (16\pi)^{\frac{3}{2}}m_{H}^{'}(\Sigma_t) = 2|\Sigma_t|^{\frac{1}{2}}\int_{\Sigma_t}(\Delta_{t}H_t+Q_{t}H_{t})\rho_td\Sigma_{t},
\end{equation*}
\noindent where
$$Q_{t}=\frac{1}{2}(R+6) + \left(\frac{4\pi}{|\Sigma_t|}-K_t\right)+\frac{1}{2}\left(|A_t|^2-\frac{1}{2|\Sigma_t|}\int_{\Sigma_t}H_t^2d\Sigma_t\right).$$
\indent Since $H_t$ is constant for each $t$, 
\begin{align*}
  (16\pi)^{\frac{3}{2}}m_{H}^{'}(\Sigma_t) & = 2|\Sigma_t|^{\frac{1}{2}}H_t\int_{\Sigma_t} Q_{t}\rho_td\Sigma_t \\
                                   & = 2|\Sigma_t|^{\frac{1}{2}}H_t\int_{\Sigma_t} (\Delta_{t}+Q_{t})(\rho_t-\overline{\rho}_t)d\Sigma_t + 2|\Sigma_t|^{\frac{1}{2}}H_{t}\overline{\rho}_{t}\int_{\Sigma_t} Q_{t}d\Sigma_{t} \\
                                   & = 2|\Sigma_t|^{\frac{1}{2}}H_t\int_{\Sigma_t} L_{t}(\rho_t-\overline{\rho}_t)d\Sigma_t + 2|\Sigma_t|^{\frac{1}{2}}H_{t}\overline{\rho}_{t}\int_{\Sigma_t} Q_{t}d\Sigma_{t}.
\end{align*}
\indent In the last line, we used the Gauss equation (\ref{eqGauss}) to see that the operator $\Delta_{t}+Q_{t}$ and the Jacobi operator $L_{t}=\Delta_{t}+Ric(N_t,N_t)+|A_t|^2$ of $\Sigma_t$ differ by a constant. \\
\indent Since $\partial M=\Sigma_0$ is minimal, the derivative of $m_H(\Sigma_t)$ is zero for $t=0$. When $t>0$, $\Sigma_t$ has positive mean curvature. Observe also that $\int_{\Sigma_t} Q_{t}d\Sigma_t \geq 0$, since $R\geq -6$ and $\Sigma_t$ is a sphere. Therefore
\begin{equation*}
  (16\pi)^{\frac{3}{2}}m_{H}^{'}(\Sigma_t) \geq 2|\Sigma_t|^{\frac{1}{2}}H_t\int_{\Sigma_t} L_{t}(\rho_t-\overline{\rho}_t)d\Sigma_t.
\end{equation*}
\indent Now we use the weak stability of $\Sigma_t$. Since $H_t$ is constant for each $t$, $L_{t}(\rho_t)= -\partial_t H_{t}$ is also constant on $\Sigma_t$. Hence the stability inequality gives \begin{equation*}
  0 \leq -\int_{\Sigma_t} L_{t}(\rho_{t}-\overline{\rho}_{t})(\rho_{t}-\overline{\rho}_t)d\Sigma_t =
    \int_{\Sigma_t} L_{t}(\overline{\rho}_{t})(\rho_{t}-\overline{\rho}_t)d\Sigma_t = 
    \overline{\rho_t}\int_{\Sigma_t} L_{t}(\rho_{t}-\overline{\rho}_t)d\Sigma_t.
\end{equation*}
\indent This implies that $m_{H}^{'}(\Sigma_t)$ is non-negative. If $m_{H}^{'}(\Sigma_t) = 0$ for $t>0$, then 
\begin{eqnarray*}
  \int_{\Sigma_t} Q_{t} d\Sigma_t=0 & \text{and} & \int_{\Sigma_t} L_{t}(\rho_t-\overline{\rho}_t)(\rho_t-\overline{\rho}_t)d\Sigma_t=0.
\end{eqnarray*}
\indent The first equality implies that $\Sigma_t$ satisfies a) and b). By the weak stability, the second equality implies that $L_{t}(\rho_t-\overline{\rho}_t)$ is constant, for it must be orthogonal to every function on $\Sigma_t$ with zero mean value. Since $L_t(\rho_t)$ is constant, this implies that $L_{t}(\overline{\rho}_t)$ is also constant. Then $c)$ follows from $a)$, $b)$ and the Gauss equation (\ref{eqGauss}). \\

\indent Once we proved the claim, inequality (\ref{penroseineq}) follows immediately:
\begin{equation*}
  \left(\frac{|\partial M|}{16\pi}\right)^{\frac{1}{2}} + 4\left(\frac{|\partial M|}{16\pi}\right)^{\frac{3}{2}} = m_{H}(\Sigma_0) \leq 
   \lim_{t\rightarrow +\infty}m_{H}(\Sigma_t) =  m.
\end{equation*}
\indent Now we analyze the equality. In this case, $m_{H}(\Sigma_t)$ must be constant and equal to $m$.
By the second part of the claim, each $\Sigma_t$ satisfies $a)$, $b)$ and $c)$
for all $t>0$. Then, possibly after a change of the parametrization $F:[0,+\infty)\times S^2 \rightarrow M$, $F^{*}g$ is a metric on $M=[0,+\infty)\times S^2$ that can be written in the form $ds^2+ V^2(s)g_{0}$, has constant scalar curvature $R=-6$, and is such that all slices $\{s\}\times S^2$ have Hawking mass $m$. These conditions uniquely characterize the metric $g_m$. And this finishes the proof.
\end{proof}
 
\noindent \textbf{Another Penrose inequality.} In Proposition \ref{geomAdSS}, we saw that the mean curvature of the coordinate spheres in the Schwarzschild anti-de Sitter spaces of mass $m>0$ is given by the function $H_{m}(r)=(2/r)\sqrt{1+r^2-2m/r}$. Observe that $H_{m}(2m)=2$ and that $H_{m}(r)>2$ for all $r>2m$. In particular, the Maximum Principle implies that there are no other closed surfaces with constant mean curvature $2$ in $([2m,+\infty)\times S^2,g_m)$. \\
\indent Let $(M,g)$ be an asymptotically hyperbolic three-manifold with connected boundary $\partial M$. Assume that $(M,g)$ has scalar curvature $R\geq -6$ and that $\partial M$ is an outermost $H=2$ surface, meaning that are no closed surfaces in $M$ with constant mean curvature $H=2$ other than $\partial M$. In this setting, the Penrose Conjecture is that the area of $\partial M$ and the total mass $m$ of $(M,g)$
are related by the inequality
\begin{equation*}
  \left(\frac{|\partial M|}{16\pi}\right)^{\frac{1}{2}} \leq m,
\end{equation*}
\noindent and that equality holds if and only if $(M,g)$ is isometric to the piece of the Schwarzschild anti-de Sitter space of mass $m$ outside the region bounded by the coordinate sphere of mean curvature $2$. \\
\indent Given $m>0$ and $M_2=[s(2m),+\infty)\times S^2$, we analogously define the space $\mathcal{M}(M_{2},m)$ of $C^3$ metrics $g$ on $M_2$ such that $\partial M_{2}=\{s(2m)\}\times S^2$ has constant mean curvature $2$ in $(M_2,g)$ and 
\begin{equation*}
 d(g,g_m) = \sup_{p\in M_2}\left(\sum_{i=0}^{3}\exp(4s)\|(\nabla^{m})^i(g-g_m)\|_{g_{m}}(p)\right) < +\infty
\end{equation*}
\indent The analogous versions of Theorem \ref{colagem} and Theorem \ref{penrose} follows immediately by the same arguments:
\begin{theorem} \label{colagem2}
  Let $m>0$. There exists $\epsilon>0$ with the following property: \\
  \indent If $g\in\mathcal{M}(M_{2},m)$ is such that $d(g,g_m)<\epsilon$, then there exists a foliation $\{\Sigma_t\}_{t\in[0,+\infty)}$ of $M_2$ such that:
  \begin{itemize}
   \item[$i)$] Each $\Sigma_{t}$ is a weakly stable CMC sphere in $(M_2,g)$, with mean curvature $H_t>2$ when $t>0$;
   \item[$ii)$] $\Sigma_{0}=\partial M_2$ is an outermost $H=2$ surface in $(M_2,g)$; and
   \item[$iii)$] $\lim_{t\rightarrow +\infty} \sqrt{\frac{|\Sigma_t|}{16\pi}}\left( 1-\frac{1}{16\pi}\int_{\Sigma_t}( H_t^2-4)d\Sigma_t \right)=m.$
  \end{itemize}
\end{theorem}
\begin{theorem}\label{penrose2}
  Given $m>0$, let $\epsilon>0$ be given by Theorem \ref{colagem2}.
  If $g\in\mathcal{M}(M_2,m)$ with $d(g,g_m)<\epsilon$ has scalar curvature $R\geq-6$, then
  \begin{equation*}
     \left(\frac{|\partial M_2|}{16\pi}\right)^{\frac{1}{2}} \leq m.
  \end{equation*}
  \indent Moreover, equality holds if and only if $(M_2,g)$ is isometric to $(M_2,g_m)$.
\end{theorem}


\bibliographystyle{amsbook}

\end{document}